\documentclass{article}

\usepackage{amsmath}
\usepackage{amssymb}
\usepackage{graphicx}
\usepackage{bm}
\usepackage{mathdots}
\usepackage{algorithm,algorithmic}

\usepackage{booktabs}
\usepackage{rotating}
\usepackage{listings}
\usepackage{enumerate}
\usepackage{tcolorbox}

\usepackage[a4paper,left=2.8cm,right=2.8cm,top=2.5cm,bottom=2.5cm]{geometry}
\usepackage{fancyhdr}
\usepackage[framemethod=tikz]{mdframed}
\newtheorem{theorem}{Theorem}[section]

\newtheorem{proposition}{Proposition}[section]
\newtheorem{lemma}{Lemma}[section]
\newtheorem{definition}{Definition}[section]
\makeatletter 
\@addtoreset{equation}{section}
\makeatother  

\newtheorem{remark}{Remark}[section]

\newenvironment{keywords}{{\noindent\bf Keywords.}}{~}

\newenvironment{proof}{{\noindent\it Proof.}\quad}{\hfill $\square$\\}

\usepackage{url}
\usepackage{mathtools}
\usepackage{lipsum}

\usepackage{stmaryrd}
\usepackage{marvosym}

\usepackage{authblk}
\usepackage{blindtext}

\title{Cellular flow control design for mixing based on the least action principle}
\author{Weiwei Hu\thanks{Department of Mathematics, University of Georgia, Athens, GA 30602, USA  (Weiwei.Hu@uga.edu)}\quad\quad Ming-Jun Lai\thanks{Department of Mathematics, University of Georgia, Athens, GA 30602, USA   (mjlai@uga.edu)}\quad\quad Hao-Ning Wu\thanks{Department of Mathematics, University of Georgia, Athens, GA 30602, USA (hnwu@uga.edu)}}

\begin{document}
\maketitle

\begin{abstract}
We consider a novel  approach for the enhancement of  fluid mixing via pure stirring strategies building upon  the Least Action Principle (LAP) for incompressible flows. The LAP is formally analogous to the Benamou--Brenier formulation of optimal transport, but imposes an incompressibility constraint. Our objective is to find a velocity field, generated by Hamiltonian flows, that minimizes the kinetic energy while ensuring that the initial scalar distribution reaches a prescribed degree of mixedness by a finite time.  This formulation leads to  a ``point-to-set"  type of optimization problem  which  relaxes the requirement on controllability of the system compared to the classic LAP framework. In particular, we assume that   the velocity field is induced by a finite set of cellular  flows that can be controlled in time.  To establish finite time feasibility, 
we introduce an operator-theoretic switching argument that combines the long-time cellular flow mixing result with the von Neumann alternating-projection theorem. 
We then leverage the direct method to establish the existence of an optimal solution.  Finally,  we derive  the  corresponding  optimality conditions for  the time-dependent control problem  and conduct numerical experiments demonstrating the effectiveness of the proposed control design.

\end{abstract}


\section{Introduction}

This work is concerned with  the bilinear control of  fluid mixing via two-dimensional cellular flows based on the Least Action Principle (LAP) for incompressible flows (e.g., \cite{brenier1989least, brenier1993dual}). Fluid mixing, the process of dispersing one material or field in another medium, is fundamental to both natural phenomena and industrial applications, from atmospheric dynamics \cite{sherwood2014spread} to microfluidic devices \cite{stroock2002chaotic}. It is governed by two primary mechanisms: {\it advection}, the transport by a fluid's bulk motion, and {\it diffusion}, the spreading due to molecular motion. 
This study focuses on advection
dominant regimes, characterized by a high Péclet number, where rapid flow transport renders diffusion negligible. Such conditions are typical in systems with high velocities, large length scales, or significant turbulence, where advective time scales are markedly shorter than diffusive ones, and hence  effectively reduce the system dynamics to those of pure transport. 
In these regimes, mixing is governed solely by advection. We accordingly address the problem of optimal mixing through  stirring or active  control  of flow advection. In particular, the laminar advection  is at the core of our investigation, where a deterministic flow is introduced.

Recent advances in the mathematical theory of mixing have provided quantitative estimates on mixing rates under various constraints. Specifically, works such as \cite{yao2017mixing, alberti2016exponential, elgindi2019universal} establish sharp bounds on mixing scales for incompressible flows. These theoretical developments complement a large body of work on stirring strategies for purely advective systems, where mixing is achieved through the design of suitable flow protocols. As explored in the pioneering work \cite{aref1984stirring} by Aref,   mixing can be achieved by inducing chaotic advection in laminar flows  through  a simple and regular flow field. In other words, complex particle trajectories can be generated by deterministic velocity fields. This phenomenon provided a mechanism for efficient fluid stretching and folding without turbulence and it was formalized in the kinematic framework in \cite{ottino1989} by Ottino.  

From a control perspective, a common approach for optimal stirring is to formulate an optimization problem for  finding an optimal velocity field or a related protocol that  enhances mixing under constraints such as fixed energy, enstrophy, or action (e.g., \cite{balasuriya2015dynamical, cortelezzi2008feasibility, couchman2010control, crippa2017cellular, eggl2020mixing, eggl2022mixing, franjione1992symmetry, gubanov2010towards, lin2011optimal, liu1994quantification, lunasin2012optimal, mathew2007optimal, sharma1997control, thiffeault2012using, vikhansky2002control, vikhansky2002enhancement}). Depending on the formulation, the objective may be to minimize a suitable mixing measure (e.g., a mix-norm or variance), promote filamentation, or accelerate the decay of large-scale structures over a finite time horizon.
Very often the velocity field is  induced by a finite set of forces or generated from a well-known reference mixer that can be controlled in time. This reflects the practical constraint that controlling flow fields arbitrarily in space is infeasible in real applications  
 (e.g., \cite{aref1984stirring, aref2002development, couchman2010control, d2002control,froyland2017optimal,  gubanov2010towards, gubanov2012cost, mathew2007optimal, ober2015multiobjective,paul2004handbook, rodrigo2003optimization, wang2003closed}). In these settings, the control is typically realized through a finite-dimensional 
parametrization, where one optimizes time-dependent coefficients associated 
with prescribed flow patterns.
 Specific strategies include optimizing switching between prescribed flows (e.g., \cite{cortelezzi2008feasibility,mathew2007optimal, ober2015multiobjective}), tuning parameters of ansatz vector fields (e.g.~\cite{gubanov2012cost,konishi2022fluid,mitchell2017designing}), and refining protocols for prescribed  mixers (e.g.~\cite{froyland2016optimal,froyland2017optimal, gubanov2010towards}). These approaches differ in how the admissible control is modeled and how the 
mixing performance is quantified, but they share the common feature of optimizing 
stirring mechanisms within constrained flow classes.
Further research incorporating   controlled flow dynamics into  the transport and mixing processes can be found in \cite{balogh2004optimal, eggl2022mixing, foures2014optimal, hu2018boundary, hu2020approximating, hu2023feedback, hu2018boundaryNS, hu2019approximating, koike2025relaxation, liu2008mixing}.

Our current work  introduces a new idea for  control of fluid mixing, inspired by the LAP for incompressible flows.  
Specifically, we aim to find an optimal  flow map induced  by cellular flows that achieves a desired   mixing efficiency  with minimal  kinetic energy.

\subsection{Motivation from LAP for incompressible flows}

 The classic LAP seeks a smooth curve connecting two given points within the group of volume- and orientation-preserving diffeomorphisms of a compact domain of a Euclidean space,  while minimizing the kinetic action. Geometrically, this corresponds to finding a minimal geodesic path in this group as addressed in the pioneering work of  Arnold \cite{arnold1966geometrie}. The Euler--Lagrange equation associated with its variational formulation  yields the Euler equations of incompressible flows  in Lagrangian form (e.g., \cite{arnold1966geometrie}).
This problem is formally analogous to the Benamou--Brenier formulation of optimal transport \cite{benamou2000computational}, which expresses the Wasserstein (Monge--Kantorovich) distance  with quadratic cost as the minimization of kinetic action without an incompressibility constraint. 

However, when incompressibility and impermeability conditions are imposed, the resulting optimal flow is governed by  Euler equations with two-point boundary values in time, which is notoriously difficult to solve  due to the lack of
a well-posed boundary value theory for such nonlinear hyperbolic systems.
In particular, the problem is highly non-convex and lacks a general well-posedness theory in the classical setting. This difficulty motivates the use of relaxed formulations, such as the generalized flow framework introduced by Brenier \cite{brenier1989least, brenier1993dual}.
 Our approach is built upon the LAP for incompressible flows, and yet  introduces  a relaxed criterion that leverages the unique features  of mixing processes. In fluid mixing, the essential objective is the degree of mixedness, a quantity not uniquely determined by any single distribution, since  different configurations can exhibit the same level of mixedness. 

 To focus on our discussion, we consider an inhomogeneous passive scalar advected by an incompressible velocity field governed by the transport equation in  an open, bounded, and connected domain $\Omega \subset \mathbb{R}^2$. Assume that the boundary  of the domain, denoted by $\Gamma$, is Lipschitz continuous,  which may include finitely many corner singularities
(e.g., polygonal domains). 
Our objective is to find a velocity field generated by cellular flows that minimizes the kinetic energy while ensuring that the initial scalar distribution reaches a prescribed degree of mixedness by a finite time $t_f > 0$. The resulting optimal control problem is formulated as follows:
\begin{align}
\min_{v}\quad&\frac{1}{2} \int_0^{t_f} \int_{\Omega} \theta(x,t)|v(x,t)|^2 \,d x\,d t,  \label{cost}\\
\text{subject to}\quad&\partial_t\theta + v\cdot \nabla \theta = 0, \label{equ:transport}\\
& \nabla \cdot v=0\quad \text{with}\quad v\cdot n|_{\Gamma}=0,   \label{equ:v}\\
&\theta(0)=\theta_0, \quad  \text{and}  \label{ini}\\
&  \|\theta (t_f)\|_{m}
 \leq r \|\theta_0\|_{m}, \label{constr_final}
\end{align}
for some given parameter $0<r<1$, where $n$ is the unit outer normal vector to the domain boundary $\Gamma$, $\theta_0\geq 0$ is the initial condition, and $\|\cdot\|_{m}$ stands for the mix-norm for quantifying mixing. 
We refer to Section~\ref{sec:mix_norm} for details on the mix-norm, 
including its definition in \eqref{equ:eta2}. The minimization in \eqref{cost} is taken over divergence-free velocity fields 
$v$ satisfying the no-penetration boundary condition $v\cdot n|_{\Gamma}=0$, 
with sufficient regularity to ensure that the transport equation \eqref{equ:transport} 
is well-posed  and that the cost functional is finite. The specific structure of $v$ is introduced in  Section \ref{cellularflow}.

This formulation reframes the search for an optimal transport map from the initial state $\theta_0$ to some  final state $\theta(t_f)$ which satisfies a set constraint, as the problem of finding a time-dependent velocity field ${v}$ that advects $\theta_0$ to some  $\theta(t_f)$ under the governing system \eqref{equ:transport}--\eqref{constr_final}, while  minimizing  the total kinetic energy. 
    Achieving the  mixing efficiency is a more natural goal than reaching a precise state. This is exerted   by requiring that the mix-norm at final time is reduced to a specified fraction of its initial value. 
    Our formulation  essentially leads to  a ``point-to-set"  type of optimization problem. This draws an essential difference from the classic LAP framework.   
Moreover, we shall employ cellular flows introduced in Section \ref{cellularflow} as mixing protocols to construct flow transport maps. This choice not only aligns with physically realizable control mechanisms but also circumvents the analytical challenges associated with infinite-dimensional control problems. 

We also note that a recent work \cite{emerick2025incompressible} develops a fluid
mixing framework using dynamic optimal transport but with a prescribed final state, where the objective is to minimize the distance between the final-time solution and a given well-mixed target.  In contrast, our formulation and mixing criterion differ substantially, as the terminal state is not prescribed but determined implicitly through the constraint.

\subsection{Mix-norm}\label{sec:mix_norm} 

A fundamental question in fluid mixing is how to quantify the degree of mixedness. A classical measure is the variance of the scalar concentration, which is equivalent to the $L^2$-norm of the scalar field \cite{danckwerts1952definition}. However, this measure is inadequate in the absence of diffusivity, as it fails to quantify the effects of pure stirring   (e.g., \cite{mathew2005multiscale, thiffeault2012using}). In fact,
due to incompressibility  and no-penetration boundary  condition in \eqref{equ:v} it can be easily  verified  that any $L^p$-norm of $\theta$ is conserved, i.e.,
\begin{align}
\|\theta(t)\|_{L^p(\Omega)}=\|\theta_{0}\|_{L^p(\Omega)}, \quad t\geq 0, \quad p\in[1,\infty]. \label{theta_Lp}
\end{align}
Mathew, Mezić, and Petzold in \cite{mathew2005multiscale}  first introduced the mix-norm based on ergodic theory, which is sensitive to both stirring and diffusion. They also showed the equivalence of the mix-norm to the $H^{-1/2}$-norm for periodic boundary conditions. In fact, any negative Sobolev norm $H^{-s}$, for $s>0$, that quantifies the weak convergence can be used as a mix-norm   (e.g., \cite{lin2011optimal, lunasin2012optimal}).
  Other measures such as the geometric mixing scale,  the measure of the interface of the mixtures, as well as entropy, are employed to quantify mixing process and the related control problems (e.g., \cite{alberti2016exponential, chakravarthy1996mixing, crippa2019polynomial, d2002control, elgindi2019universal, vikhansky2002control, vikhansky2002enhancement, yao2017mixing}). A more detailed review can be found in \cite{thiffeault2012using}.

 In this work, we consider a general open bounded and connected domain for the scalar field with no additional boundary conditions, where no-penetration is imposed on the velocity field. Instead of directly using the negative Sobolev norm, we employ the dual norm $\|\cdot\|_{(H^{s}(\Omega))'}$ for $s>0$ as a mix-norm, where $(H^{s}(\Omega))'$ is the dual space of $H^{s}(\Omega)$. Without loss of generality, we adopt $(H^{1}(\Omega))'$ as the \emph{mix-norm}, following our earlier works in\cite{hu2018boundary,hu2018boundaryNS,hu2020approximating,hu2019approximating}.
 
To better characterize the dual norm, we consider the  elliptic boundary value problem:
\begin{equation}\label{equ:eta}
(-\Delta + I)\eta = \theta \quad \text{in } \Omega, \qquad \frac{\partial \eta}{\partial n}\big|_{\Gamma} = 0.
\end{equation}
The mix-norm is then defined as
\[
\|\theta\|_{(H^1(\Omega))'} = \|\eta\|_{H^1(\Omega)}.
\]
Let $\mathcal{A}=-\Delta+I$ with domain $D(\mathcal{A})=\{\eta\in H^2(\Omega)\colon \frac{\partial \eta}{\partial n}|_{\Gamma}=0\}$.  Then
\begin{equation}\label{equ:eta2}
\|\theta\|^2_{(H^1(\Omega))'}=(\mathcal{A}^{-1}\theta, \theta),
\end{equation}
where $(\cdot,\cdot)$ denotes the standard $L^2(\Omega)$-inner product.

Further let $\bar{\theta}_0 := \frac{1}{|\Omega|} \int_{\Omega} \theta_0(x) dx$ denote the spatial average of the initial scalar field. Since the spatial average is conserved for all time, i.e.,
\[ \frac{1}{|\Omega|}\int_{\Omega}\theta(x,t)dx=\bar{\theta}_0, \quad \forall t>0,
\] 
without loss of generality, we can replace the final time condition  \eqref{constr_final} as
 \begin{align}
      \|\theta(t_f)-\bar{\theta}_0\|^2_{(H^1(\Omega))^{\prime}}\leq r^2\left\|\theta_0-\bar{\theta}_0\right\|^2_{(H^1(\Omega))^{\prime}}, \quad 0<r<1.\label{2constr_final}
 \end{align}

\subsection{Cellular flows based finite dimensional control design}\label{cellularflow}

In this work, we employ  cellular flows  as   two-dimensional mixing protocols,  which 
play a fundamental role in mathematical analysis of fluid mixing  and  engineering applications (e.g., \cite{bedrossian2017enhanced, zelati2024mixing,crippa2019polynomial, crippa2017cellular, elgindi2019universal, fannjiang1994convection, franjione1992symmetry, heinze2001diffusion, iyer2021convection, iyer2023quantifying, mathew2007optimal, novikov2005boundary, yao2017mixing}).
  Specifically, consider  that  the flow velocity in \eqref{equ:transport}  is induced by a set of finite basis flows in a two-dimensional  open unit square $\Omega=(0,1)^2$:
\begin{equation}
v(x,t) = \sum_{i=1}^N u_{i}(t)  b_{i}(x),\label{v_finite}
\end{equation}
where $b_i(x), i=1, 2,\dots, N,$   are the cellular flows generated by the  Hamiltonians $H_i(x_1,x_2)=\frac{1}{i\pi}\sin  (i \pi x_1)\sin (i\pi x_2)$
\begin{align}
    b_i=\nabla^{\perp}H_i(x_1,x_2)
= \begin{bmatrix}
      -\sin (i \pi x_1)\cos(i \pi x_2)          \\[0.3em]
        \cos (i \pi x_1)\sin(i \pi x_2)      
     \end{bmatrix},
\label{cellular_2D} 
\end{align} 
and the temporal coefficients $u_i(t), i=1,2,\dots,N,$ serve as the control inputs. 
It is clear that
\begin{equation}\label{equ:basisflows}
\nabla \cdot b_{i} = 0 \quad \text{and} \quad b_{i} \cdot n|_{\Gamma} = 0.
\end{equation} 
Fig. \ref{fig:displayonsquare} displays cellular flows $b_i$ for $i=1,2,3,4$,  which comprise  $i^2$ cells. 

\begin{figure}[htbp]
\centering
\includegraphics[width =0.48\textwidth]{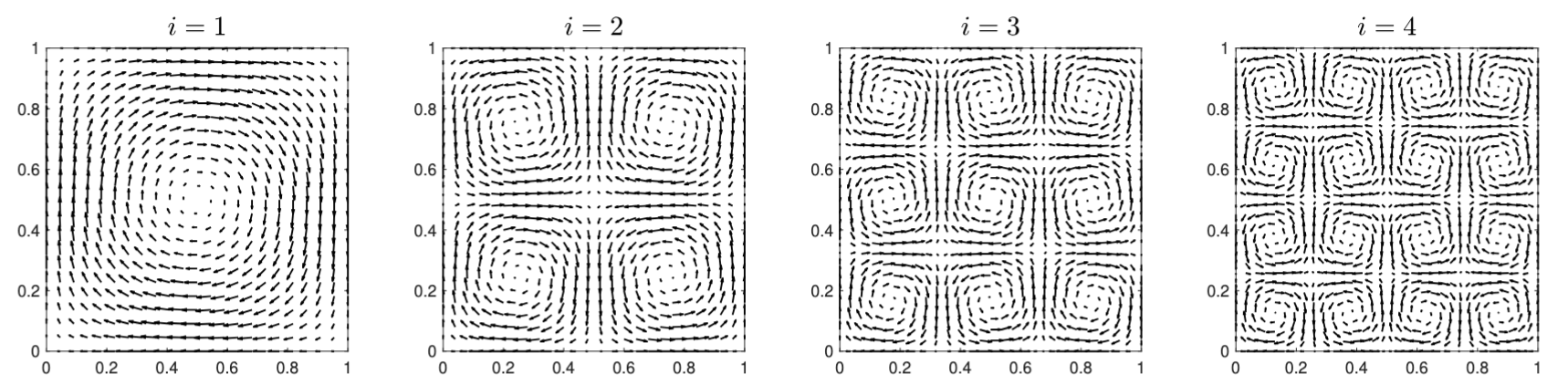}
\caption{Cellular flows on the square.}\label{fig:displayonsquare}
\end{figure}  

Cellular flows generated by  Hamiltonians are incompressible, periodic flow patterns in a two-dimensional plane that consist of repeating lattices of swirling, closed-loop trajectories. 
 They are  especially valuable in regimes where turbulence is unattainable, such as at low Reynolds numbers, or where precise fluid manipulation is required. These protocols are widely employed across scales, from microfluidic devices to industrial-scale reactors, often implemented through designed system geometries (e.g., ridges, baffles, or electrodes) or time-dependent forcing strategies (e.g., flow reversal or rotating fields) that deterministically produce cellular flow structures.
For instance, staggered herringbone mixers use patterned grooves on microchannel walls to induce chaotic advection via asymmetric, staggered cellular flows, making them essential for efficient reagent blending in lab-on-a-chip diagnostics. On a larger scale, stirred tank reactors employ rotating impellers to generate intense trailing vortices (small-scale cellular flows) that homogenize contents in chemical and biochemical processes. A central aim of these designs is to transition from orderly cellular flows into chaotic advection, achieving exponential stretching of material lines and ultimately enabling highly efficient mixing.

Mathematically, our formulation retains the quadratic action of the
Benamou--Brenier framework, but not its convexity (see Remark \ref{nonconvexity}).
The feasible set, and hence the reduced
optimization problem, is generally nonconvex. Nevertheless, provided that the feasible set is nonempty, we prove the existence of a global
minimizer using the direct method. Establishing this nonemptiness is therefore a central analytical challenge, particularly within a
finite-dimensional control space.

 Recent progress in the mathematical analysis of mixing by incompressible flows, particularly cellular flows,
\cite{brue2024enhanced,zelati2024mixing,crippa2017cellular,
elgindi2019universal,iyer2023quantifying,yao2017mixing}
provide the broader theoretical background for our feasibility analysis. 
Of particular relevance to the present work is the mixing estimate of
\cite[Theorem~3]{brue2024enhanced}, which we combine with amplitude
scaling to establish a preliminary feasibility result for a restricted
class of initial data. We then construct a
switching control that alternates  between cellular flows to obtain a more general feasibility result.
The details are presented in Section~\ref{sec:feasibility}.

Relatedly, \cite{kiselev2008relaxation} characterizes a time-periodic
incompressible flow as relaxation enhancing if and only if its period
operator associated with the flow has no nonconstant \(H^1\) eigenfunctions. This supports the
role of time dependence in removing autonomous invariant obstructions,
although this spectral criterion is not implied by the  common kernel 
condition (see \eqref{ker_condition}) used in our switching construction.

\subsection{Organization}

The following sections detail the analysis and implementation of our control design. In the next section, we first establish the feasibility of the problem by showing that the constraint set defined by \eqref{equ:transport}--\eqref{ini} and \eqref{2constr_final} together with \eqref{v_finite}--\eqref{cellular_2D} is nonempty. We then show in Section \ref{sec:optimality_conditions} the existence of a global optimal solution and present the corresponding optimality conditions. The numerical implementation of our control design, including a solution algorithm, is discussed in Section \ref{sec:implementation}. Finally, Section \ref{sec:numerical_results} validates our approach with numerical results on a unit square and Section  \ref{conclusion} concludes the paper.

In the sequel, the symbol $C$ denotes a generic positive constant, which is allowed to depend on the domain as well as on indicated parameters.

\section{Feasibility of the problem}\label{sec:feasibility}

The feasibility of steering the system to a prescribed level of mixedness at a given final time is closely tied to the bilinear controllability of the underlying transport system, a property that becomes particularly restrictive when the velocity field is confined to a finite-dimensional basis.
A related difficulty was also observed in the work of Mathew et al.~\cite{mathew2007optimal}, where enforcing a prescribed level of mixing at a fixed final time would, in general, require strong controllability properties, particularly under finite-dimensional actuation constraints. To circumvent this difficulty, they adopt an alternative formulation in which the mix-norm is minimized at a given time subject to a constraint on the total action (i.e., the time-integrated kinetic energy). In this framework, the achievable level of mixing is not prescribed a priori, but is instead determined implicitly by the available control authority and the imposed energy budget.

In contrast, our formulation adopts a ``point-to-set'' perspective, where the objective is to steer the system into a target set characterized by a prescribed level of mixedness. This relaxes the requirement of exact controllability to a given state and instead focuses on reaching a neighborhood in the state space, thereby allowing for greater flexibility under finite-dimensional control constraints.

This work establishes that, for any initial scalar distribution satisfying the regularity assumptions stated below, there exists a time-dependent velocity field within the finite-dimensional space spanned by the cellular flows such that the prescribed final-time
mixing constraint is satisfied. 

Recall the following definition from \cite{brue2024enhanced}, with a  modification to accommodate general bounded domains and nonnegative initial data.

\begin{definition}\label{def1}
Let $\gamma\colon[0,\infty)\to[0,\infty)$ be continuous and
nonincreasing, with $\gamma(t)\to0$ as $t\to\infty$. A
time-independent Hamiltonian velocity field $b=\nabla^{\perp}H$ satisfying
\eqref{equ:v} is said to be \emph{mixing with rate} $\gamma$ if, for   $\theta_0\in H^1(\Omega)$  satisfying 

\begin{equation}
\tag{H}
\theta_0-\bar\theta_0
\in\bigl(\ker(b\cdot\nabla)\bigr)^{\perp},
\end{equation}
where the orthogonal complement is taken in $L^2(\Omega)$, the solution of  \eqref{equ:transport} with
$\theta(0)=\theta_0$ satisfies
\begin{align}
\|\theta(t)-\bar\theta_0\|_{(H^1(\Omega))'}
\leq
\gamma(t)\|\theta_0-\bar\theta_0\|_{H^1(\Omega)},
\quad t\geq0.
 \label{def_mixing_rate}
\end{align}
\end{definition}

For the cellular Hamiltonian field $b$,  condition {\rm (H)} is equivalent 
 to requiring that $\theta_0-\bar\theta_0$ has zero  time average along almost every regular closed streamline of $b$. This is precisely the
zero streamline-average condition used in
\cite[Definition~1.2]{brue2024enhanced}. 
This equivalence is proved in Lemma~\ref{lem_equivalence} in
Appendix~\ref{appendix_A}. Proposition~\ref{prop:symmetric_initial_data}
further provides a class of reflection-antisymmetric initial data
satisfying condition {\rm(H)}.

By \cite[Theorem~3]{brue2024enhanced}, together with a spatial
rescaling and the reflection symmetries of the standard cellular flow,
we know that $b_1$ is mixing on $\Omega=(0,1)^2$. More precisely, for
every $0<\epsilon<1/3$, its mixing rate satisfies
\begin{equation}\label{stability_one_flow}
\gamma(t)
\leq
C(\epsilon)\,t^{-\frac{1}{3}+\epsilon},
\qquad t>0,
\end{equation}
where $C(\epsilon)>0$ is a constant depending on  $\epsilon$.

 Next we apply \eqref{stability_one_flow} with an amplitude scaling  to establish the
feasibility of \eqref{cost}--\eqref{ini} and
\eqref{2constr_final} for initial data satisfying condition
{\rm(H)}, using a single cellular flow. 
We then establish our main
feasibility result for general initial data by constructing controls
that alternate between cellular flows.

 For notational convenience, we  let
\[\mathbf{u}(t) = [u_1(t), u_2(t), \ldots, u_N(t)]^T\]
and set $U_{ad}=(L^2(0, t_f))^N$ as  the set of admissible controls.

\begin{proposition}[Feasibility by amplitude scaling]
\label{prop:feasibility}
Assume that $ \theta_0\in H^1(\Omega)$  and 
$
\theta_0-\bar\theta_0
\in\bigl(\ker(b_1\cdot\nabla)\bigr)^\perp.
$
Then, for every $t_f>0$ and $0<r<1$, there exists $\alpha>0$ such
that the control $\mathbf{u}=(u_1,\ldots,u_N)$ defined by
\begin{equation}\label{constant_control}
u_1(t)\equiv\alpha,
\qquad
u_i(t)\equiv0,\quad i=2,\ldots,N,
\end{equation}
belongs to $U_{\rm ad}$, and the corresponding solution of \eqref{equ:transport}--\eqref{ini} satisfies
the terminal constraint \eqref{2constr_final}. Consequently, the
feasible set is nonempty.
\end{proposition}

\begin{proof}
Let $\widehat\theta(s)$ denote the solution generated by the
unit-amplitude cellular field $b_1$:
\[
\partial_s\widehat\theta+b_1\cdot\nabla\widehat\theta=0,
\qquad
\widehat\theta(0)=\theta_0.
\]
Thus by  \eqref{def_mixing_rate}--\eqref{stability_one_flow}, there exists $T>0$ such that
\begin{equation}\label{1EST_mixingrate}
\|\widehat\theta(T)-\bar\theta_0\|_{(H^1(\Omega))'}
\leq
r\|\theta_0-\bar\theta_0\|_{H^1(\Omega)}.
\end{equation}
Set $\alpha=T/t_f$ and define
$
\theta(x,t)=\widehat\theta(x,\alpha t).
$
Using the change of variables $s=\alpha t$, we obtain
\[
\partial_t\theta+\alpha b_1\cdot\nabla\theta
=
\alpha\bigl(
\partial_s\widehat\theta+b_1\cdot\nabla\widehat\theta
\bigr)=0,  \quad \theta(x,0)=\theta_0,
\]
and $\theta(x,t_f)=\widehat\theta(x,T)$ satisfies \eqref{2constr_final}.
 Thus $\theta$ is the   state associated with the control \eqref{constant_control}
and $\theta(x,0)=\theta_0$.
Moreover, since 
$
\|\mathbf{u}\|_{U_{\rm ad}}^2=\alpha^2t_f<\infty
$, 
$\mathbf{u}\in U_{\rm ad}$ and hence $\theta$ is a feasible solution to  \eqref{equ:transport}--\eqref{ini} and \eqref{2constr_final}.
\end{proof}

Proposition~\ref{prop:feasibility} establishes feasibility for a single flow under condition {\rm(H)}. In fact, trajectories under $b_1$ remain on closed level sets of $H_1$, and components in
$\ker(b_1\cdot\nabla)$ remain invariant regardless of the amplitude or duration of the control.

This obstruction motivates switching between cellular flows with different streamline geometries. The resulting nonautonomous flow
need not preserve the streamline foliation of either individual field and can therefore establish feasibility for a broader class of initial data.

\begin{theorem}[Feasibility by alternating cellular flows]
\label{thm:switching_feasibility}
Assume that $N\geq2$, and let $b_I$ and $b_J$ be two distinct fields
in the cellular basis such that
\begin{align}
\ker(b_I\cdot\nabla)\cap\ker(b_J\cdot\nabla)
=
\operatorname{span}\{1\}. \label{ker_condition}
\end{align}
Let $\theta_0\in H^1(\Omega)$. Then, for every
$t_f>0$ and $0<r<1$, there exists a piecewise-constant control
$\mathbf{u}\in U_{\rm ad}$ that alternately activates $b_I$ and $b_J$. The
corresponding solution of \eqref{equ:transport}--\eqref{ini} satisfies
the terminal constraint \eqref{2constr_final}. Consequently, the
feasible set is nonempty.
\end{theorem}

The proof is given in Appendix~\ref{appendix_B}, where we also verify that the pair $(b_1,b_2)$ satisfies the common kernel condition  \eqref{ker_condition}. Note that the nonnegativity of the initial datum is not required for feasibility. Unlike the single flow result, 
 Theorem~\ref{thm:switching_feasibility} does not require the initial  datum to satisfy condition {\rm(H)} with respect to either $b_I$ or  $b_J$. The result is qualitative: the switching times and amplitudes  may depend on the initial datum and the prescribed mixing level, and no uniform mixing-rate or control cost estimate is asserted.
Since the prescribed mixing level $r>0$ can be chosen arbitrarily small, the theorem establishes approximate controllability to the mean  in the mix-norm, which we refer to as \emph{approximate mixing 
 controllability}.

\section{Existence and the first-order optimality conditions}\label{sec:optimality_conditions}

In this section, we establish the existence of an optimal control for
the finite-dimensional control problem and derive the corresponding
first-order necessary optimality conditions. Although $H^1$ regularity suffices for our analysis, we additionally assume $\theta_0\in L^\infty(\Omega)$ to reflect the physically natural boundedness of the initial scalar, as is standard in the fluid mixing literature.

\begin{theorem}[Existence of an optimal control]
Let $\theta_0\in L^\infty(\Omega)\cap H^1(\Omega)$ be nonnegative and $\theta_0\not\equiv 0$. 
For any $t_f>0$ and $0<r<1$,
if the feasible set defined   by
\eqref{equ:transport}--\eqref{ini} and \eqref{2constr_final} is nonempty,
then the optimal control problem \eqref{cost} admits at least one global 
minimizer $u^*\in U_{\rm ad}$. 
\end{theorem}

\begin{proof}
We sketch the proof based on the standard   direct method. The detailed argument follows 
\cite{hu2018boundary,hu2020approximating}. 

 For every $\mathbf{u}\in U_{\rm ad}$, the smoothness
of the cellular fields implies
\begin{align}
v=\sum_{i=1}^N u_i b_i
\in L^1(0,t_f;W^{1,\infty}(\Omega)), \label{reg_v}
\end{align}
and hence the transport equation \eqref{equ:transport} admits a unique solution
\begin{align}
\theta\in
L^\infty\bigl(0,t_f;L^\infty(\Omega)\cap H^1(\Omega)\bigr), \label{reg_theta}
\end{align}
and the kinetic energy defined in \eqref{cost} is finite. 

Moreover, since $\theta_0\geq0$ is nontrivial, there exist $\delta>0$ and
$E_0\subset\Omega$ with $|E_0|>0$ such that $\theta_0\geq\delta$
on $E_0$. Its transported image $E(t)$ satisfies
$|E(t)|=|E_0|$ and $\theta(\cdot,t)\geq\delta$ on $E(t)$.
Analyticity and linear independence of $b_i$  together with 
compactness of the unit sphere in $\mathbb R^N$, imply that for every nonzero $a\in \mathbb R^N$, there exists $c_0>0$ such that
\begin{align}
\int_E\left|\sum_{i=1}^N a_i b_i(x)\right|^2dx
\geq c_0 |a|^2
\quad\text{whenever } |E|=|E_0|. \label{J_coercivity}
\end{align}
Consequently,
$J(\theta,\mathbf u)
\geq \frac{\delta c_0}{2}
\|\mathbf u\|^2_{U_{\rm ad}},
$
thus $J$ is coercive on $U_{\rm ad}$.

Further note that the control-to-state map is sequentially continuous
from weak convergence in $U_{\rm ad}$ to strong convergence in $C([0,t_f];L^2(\Omega))$. Since the feasible set is nonempty by Theorem~\ref{thm:switching_feasibility}, a minimizing sequence exists. Coercivity implies that it is bounded in $U_{\rm ad}$ and therefore
admits a weakly convergent subsequence. The corresponding
strong convergence of the states shows that the weak limit
remains feasible. 

Finally, the weak lower semicontinuity of the cost implies that this
limit is an optimal control. 
\end{proof}


\begin{remark} [On the Benamou--Brenier formulation]\label{nonconvexity}
For an admissible state-control pair, define
$
m=\theta v.
$
Although the transport equation becomes the linear continuity equation
\[
\partial_t\theta+\nabla\cdot m=0,
\]
and the perspective integrand $|m|^2/\theta$, understood to be zero
when $(\theta,m)=(0,0)$ and $+\infty$ when $\theta=0$ and $m\neq0$,
is jointly convex in $(\theta,m)$, this change of variables does not establish
convexity of the present problem. Indeed, the original admissible
class also requires
\[
\frac{m}{\theta}\in\operatorname{span}\{b_1,\ldots,b_N\}
\]
on $\{\theta>0\}$, with coefficients depending only on time and not
on space. This condition is not generally preserved under convex
combinations. Therefore, the convexity of the standard
Benamou--Brenier formulation does not carry over to the present
finite-dimensional control problem.
\end{remark}


To derive the first-order necessary optimality conditions, we employ a
Lagrange-multiplier approach. The regularity of $(v,\theta)$ given by
\eqref{reg_v}--\eqref{reg_theta} implies the Gâteaux differentiability
of the control-to-state map $u\mapsto\theta$ and, consequently, of the
cost functional (e.g.~\cite{hu2018boundary,hu2020approximating}). Under a suitable constraint qualification, every
locally optimal feasible pair must satisfy the following conditions.

\begin{theorem}\label{thm:optimalitycondition}
Let $\theta_0\in L^\infty(\Omega)\cap H^1(\Omega)$ be nonnegative
and nonconstant, and fix $t_f>0$ and
$0<r<1$. Suppose that $(u,\theta)$ is a locally optimal feasible
pair and that, whenever the terminal constraint is active, its
derivative with respect to $u$ is nonzero. Then there exist
\[
\rho\in L^\infty\bigl(0,t_f;L^\infty(\Omega)\cap H^1(\Omega)\bigr)
\quad\text{and}\quad \lambda\geq0,
\]
such that
\begin{enumerate}
    \item \textbf{State Equation}:
    \begin{equation}\label{equ:stateequation}
    \partial_t\theta + \sum_{i=1}^N u_i (b_i\cdot \nabla \theta ) = 0,\quad \theta(x,0)=\theta_0(x),
    \end{equation}
    with  the final-time constraint:
      \begin{equation}\label{equ:feasibility} 
  \|\theta(t_f)-\bar{\theta}_0\|_{(H^1(\Omega))^{\prime}}^2 \leq r^2C_0^2;
      \end{equation}
where $C_0:=\|\theta_0-\bar{\theta}_0\|_{(H^1(\Omega))^{\prime}};$
    \item \textbf{Adjoint Equation}:
    \begin{equation}\label{equ:adjointequation}
    \begin{split}
    -&\partial_t\rho -  \sum_{i=1}^N u_i (b_i\cdot\nabla\rho )
    = -\frac{1}{2} \left|\sum_{i=1}^N u_i b_i\right|^2 ,\\
    & \rho(t_f)
    =-2\lambda\mathcal{A}^{-1}(\theta(t_f)-\bar{\theta}_0),
    \end{split}
    \end{equation}
    where $\mathcal{A}^{-1}(\theta(t_f)-\bar{\theta}_0)$ is defined by the elliptic problem \eqref{equ:eta};
    \item \textbf{Optimality  Condition}:
    \begin{equation}\label{equ:optimalcontrol}
  \mathbf{u}(t) =\bm{M}^{-1}(t)\bm{p}(t),
  \end{equation}
where $\bm{M}(t)\in\mathbb{R}^{N\times N}$ is a matrix with entries  $M_{ij}(t)\colon= \int_\Omega \theta b_i\cdot b_j\,dx$ and $\bm{p}(t)\in\mathbb{R}^N$ is a vector with entries $p_i(t):= \int_{\Omega} (b_i \cdot \nabla \rho )  \theta\,dx$;
    \item \textbf{Complementary Slackness}:
    \begin{equation}\label{equ:complementaryslackness}
    \lambda\left( \|\theta(t_f)-\bar{\theta}_0\|_{(H^1(\Omega))^{\prime}}^2 - r^2C_0^2\right) = 0, \quad \lambda\geq0.
    \end{equation}
\end{enumerate}
\end{theorem}

\begin{proof}
First we introduce the Lagrangian 
\begin{equation}
\begin{aligned}
&\mathcal{L}(\theta,\mathbf{u},\theta(t_f); \rho,\lambda)\\
=& J(\theta, \mathbf{u}) +\underbrace{\int_0^{t_f} \int_\Omega \rho\big(\partial_t\theta 
+  \sum_{i=1}^N u_i (b_i\cdot \nabla \theta )\big) \,dx\,dt}_{\text{PDE constraint}} \\
& + \underbrace{\lambda\big( \|\theta(t_f)-\bar{\theta}_0\|_{(H^1(\Omega))^{\prime}}^2 - r^2C_0^2\big)}_{\text{Terminal constraint}},
\quad \lambda\ge0,
\end{aligned}
\label{Lagrangian}
\end{equation}
where $\rho$ is the Lagrange multiplier or the adjoint state corresponding  to $\theta$ and $\lambda \geq 0$ is the Karush--Kuhn--Tucker (KKT) multiplier. 
By Stokes formula and  conditions  given by \eqref{equ:basisflows} we have
\begin{align}
&\int_\Omega \rho \left( \sum_{i=1}^N u_i (b_i \cdot  \nabla \theta )\right)\,dx
=\sum_{i=1}^N u_i \int_\Omega \rho\nabla \cdot \left( b_i\theta\right) \,dx\nonumber\\
=& \sum_{i=1}^N u_i\left( \int_{\Gamma}\rho (b_i \theta)\cdot n  \,ds 
-  \int_{\Omega}  \nabla \rho \cdot (b_i \theta)\,dx\right)\nonumber\\
=& -\sum_{i=1}^N u_i \int_{\Omega} ( b_i \cdot  \nabla \rho) \theta\,dx. \label{Stokes_formu}
\end{align}

Applying integration by parts to the second term on the right hand side of \eqref{Lagrangian} together with \eqref{Stokes_formu} and \eqref{equ:eta2} follows
\begin{equation*}\begin{split}
&\mathcal{L}(\theta,\mathbf{u},\theta(t_f); \rho,\lambda)  \\
=&\frac{1}{2} \int^{t_f}_0   \int_{\Omega} \theta   \left|\sum_{i=1}^N u_i b_i\right|^2 \,dx dt+(\rho(t_f), \theta(t_f))-\left(\rho(0), \theta_0\right)\\
&+\int_0^{t_f}\int_{\Omega}(-\partial_t \rho \theta )\,dx d t-\sum_{i=1}^N \int_0^{t_f}u_i \Big(\int_{\Omega} (b_i \cdot \nabla \rho )  \theta\,dx \Big)\,d t\\
&+\lambda \left(\left(\mathcal{A}^{-1} (\theta(t_f)-\bar{\theta}_0), \theta(t_f)-\bar{\theta}_0\right) - r^2C_0^2\right).
\end{split}\end{equation*}
The adjoint state $\rho$ is chosen such that the first variation  of $\mathcal{L}$ with respect to $\theta$ vanishes, i.e., $\frac{\partial \mathcal{L}}{\partial \theta}=0$, which results in \eqref{equ:adjointequation}. The final-time condition \eqref{equ:adjointequation} is derived by setting $\frac{\partial \mathcal{L}}{\partial \theta(t_f)}=0$. 
Now setting  $ \frac{\partial \mathcal{L}}{\partial u_i}=0$ for each $i=1,2, \dots N$,  follows
\[ \sum^N_{j=1}u_j\Big( \int_{\Omega}   \theta b_i\cdot  b_j  \,dx\Big)-\int_{\Omega} (b_i \cdot \nabla \rho )  \theta\,dx  =0.\]
Let  $\bm{M}(t)\in\mathbb{R}^{N\times N}$ be a  matrix with  entries  $M_{ij}(t)\colon= \int_\Omega \theta b_i\cdot b_j\,dx$ and $\bm{p}(t)\in\mathbb{R}^N$ is a vector with entries $p_i(t):= \int_{\Omega} (b_i \cdot \nabla \rho )  \theta\,dx$. 
By \eqref{J_coercivity},
$\bm M(t)$ is uniformly positive definite and hence invertible,
which yields the optimality condition \eqref{equ:optimalcontrol}.

Lastly,  we require that the KKT multiplier $\lambda$ for the  final-time constraint satisfies  the complementary slackness condition \eqref{equ:complementaryslackness}.
If the optimal KKT  multiplier $\lambda > 0$, then the constraint \eqref{2constr_final} is active, i.e., the inequality is satisfied as an equality at the optimum. The detailed explanation  on a Lagrangian-based view of the adjoint approach can be found in \cite[p.\,63, 1.6.4]{hinze2008optimization}. 
\end{proof}


\section{Numerical implementation}\label{sec:implementation}

The optimality conditions derived in Theorem \ref{thm:optimalitycondition} can be solved  via the fixed-point iteration 
\begin{equation*}
\begin{bmatrix}
u^{k+1}\\
\lambda^{k+1}
\end{bmatrix}
=\begin{bmatrix}
F_1(u^k,\lambda^k)\\
F_2(u^k,\lambda^k)
\end{bmatrix},
\end{equation*}
where the mapping $F_1$ encodes the coupled system given by the state equation \eqref{equ:stateequation}, the adjoint equation \eqref{equ:adjointequation}, and the optimality  condition \eqref{equ:optimalcontrol}, and  the mapping $F_2$ incorporates both the final time condition \eqref{equ:feasibility} and the complementary slackness condition \eqref{equ:complementaryslackness}. Specifically,  conditions \eqref{equ:feasibility} and \eqref{equ:complementaryslackness} can be satisfied  using the  following fixed-point iteration 
\begin{equation}\label{equ:updatelambda}
\lambda^{k+1} = \max\left(\lambda^k + \beta^k \left( \|\theta^k(t_f)-\bar{\theta}_0\|_{\left(H^1(\Omega)\right)^{\prime}}^2 - r^2C_0^2\right),0\right),
\end{equation}
where $\beta^k > 0$ is a step-size parameter. At convergence, we recover
$$\lambda^* = \max\left(\lambda^* + \beta^* \left( \|\theta^*(t_f)-\bar{\theta}_0\|_{\left(H^1(\Omega)\right)^{\prime}}^2 - r^2C_0^2\right),0\right)$$
for some $\beta^*>0$, corresponding to the step-size parameter in the limit of the iteration, which implies $\lambda^* \geq 0 $ and 
 \begin{equation*} \begin{cases}
        \text{if } \|\theta^*(t_f)-\bar{\theta}_0\|_{\left(H^1(\Omega)\right)^{\prime}}^2 - r^2C_0^2< 0\text{, then } \lambda^* = 0, \\
        \text{if } \lambda^* > 0\text{, then }\|\theta^*(t_f)-\bar{\theta}_0\|_{\left(H^1(\Omega)\right)^{\prime}}^2 - r^2C_0^2 = 0,
        \end{cases}
    \end{equation*}
    precisely satisfying conditions \eqref{equ:feasibility} and \eqref{equ:complementaryslackness}. In particular, when updating $u^{k+1}$, we introduce a relaxation factor $\alpha^{k+1}\in(0,1]$ to control the update:
\begin{equation}\label{equ:updateu}
u^{k+1} = (1-\alpha^{k+1})u^k + \alpha^{k+1}\bm{M}_{k+1}^{-1}\bm{p}_{k+1},
\end{equation}
where  $\bm{M}_{k+1}\in\mathbb{R}^{N\times N}$ and  $\bm{p}_{k+1}\in\mathbb{R}^N$ with  entries $\int_\Omega \theta^{k+1} b_i\cdot b_j\,dx$ and $\int_\Omega \theta^{k+1}\nabla\rho^{k+1}\cdot b_i\,dx$, respectively. This relaxation enhances stability by preventing overshooting and oscillations between iterations.

Let
\[\mu^{k+1}= \|\theta^{k+1}(t_f)-\bar{\theta}_0\|_{(H^1(\Omega))^{\prime}}^2 - r^2C_0^2\]
be the final-time constraint violation for iterating  $\theta^{k+1}$. Let $0<\varepsilon_1,\varepsilon_2\ll 1$ be stopping criteria parameters. We propose the following Algorithm \ref{alg:FPI} to solve the first-order optimality conditions governed by \eqref{equ:stateequation}--\eqref{equ:complementaryslackness}. For spatial discretization, we adopt the spline collocation method (e.g., \cite{lai2022multivariate,lai2023trivariate}), with bivariate splines serving as the basis functions for approximating numerical solutions in space (e.g., \cite{lai2007spline,awanou2005multivariate}). For the temporal aspect, the implicit Euler scheme is used for the state equation \eqref{equ:stateequation}, while the explicit Euler scheme is applied to the (backward) adjoint equation \eqref{equ:adjointequation}.

\begin{algorithm}[H]
\caption{Fixed-point iteration for the optimal control problem \eqref{cost}--\eqref{ini} and \eqref{2constr_final}}
\label{alg:FPI}
\begin{algorithmic}
\STATE {\textbf{While }}{\small$\mu^{k+1}> \varepsilon_1$ \emph{or} $|J(u^{k+1}) - J(u^k)|/|J(u^k)|>\varepsilon_2$ }
\STATE \hspace{0.1cm} Solve the state equation \eqref{equ:stateequation} with $u^k$ for $\theta^{k+1}$\\
\STATE \hspace{0.1cm}  Solve the equation \eqref{equ:eta} with $\theta^{k+1}$ for $\eta^{k+1}$\\
\STATE \hspace{0.1cm}  Update $\beta^{k+1}$ and update $\lambda^{k+1}$ via \eqref{equ:updatelambda}\\
\STATE \hspace{0.1cm} Solve the adjoint equation \eqref{equ:adjointequation} with $\theta^{k+1}(t_f)$ for $\rho^{k+1}$\\
\STATE \hspace{0.1cm}  Update $\alpha^{k+1}$ and update $u^{k+1}$ via \eqref{equ:updateu}
\end{algorithmic}
\end{algorithm}

\section{Numerical results}\label{sec:numerical_results}

In this section, we present numerical results supporting the proposed control framework for achieving a prescribed level of mixing within a finite time horizon.
As illustrated in Fig. \ref{fig:displayonsquare}, the streamlines of the cellular flows within each cell form closed circular patterns and are tangent to, but do not cross, those of adjacent cells. This structure prevents particles from accessing the entire domain by trapping them in localized regions.  To address this limitation, we explore the use of linear combinations of  cellular flows of different frequency. The key idea is that if a particle is trapped in one cell, activating a different flow can transport it across the cell boundaries to achieve  more global mixing.  In our numerical experiments,
  we investigate three different combinations: $i = 1~\&~2$, $i = 1~\&~3$, and $i = 1~\&~4$, all simulated up to a final time of $t_f=1$. Here $i$ denotes the index of the cellular basis functions introduced in \eqref{v_finite}--\eqref{cellular_2D}.

For the numerical solution of PDEs, we employ a spline-based solver with a mesh size (i.e., the length of the longest edge in the triangulation) of $h = 0.0884$ and a time step of $\tau = 0.005$. The optimization algorithm is initialized with $\lambda^0 = 1$, and the convergence criteria are set to $\varepsilon_1=\texttt{5e-4}$ and $\varepsilon_2=\texttt{1e-3}$. Additionally, we implement adaptive update strategies for $\alpha^k$ and $\beta^k$:
\begin{equation*}
\alpha^{k+1} = \begin{cases}
\max(0.5\alpha^k,0.05) &\text{if }\mu^{k+1}<\varepsilon_1,\\
\alpha^k&\text{otherwise},
\end{cases}
\end{equation*}
and 
\begin{equation*}
\beta^{k+1} = \begin{cases}
\frac{2J(u^k,\theta^{k+1})}{\mu^{k+1}} &\text{if }\mu^{k+1} \geq\varepsilon_1\text{ and }J(u^k,\theta^{k+1})\geq 1,\\
100 &\text{if }\mu^{k+1} \geq\varepsilon_1\text{ and }J(u^k,\theta^{k+1})< 1,\\
250 &\text{if }\mu^{k+1}<\varepsilon_1.\\
\end{cases}
\end{equation*}

For the control parameters, we set $r = 0.3$ in the final time constraint \eqref{2constr_final} and the initial control inputs as  $u_1(0)= 0$ and $u_i(0)=1$, $i=2,3,4$, for each combination,  which initially activate the multi-cell flows while keeping the single-cell flow idle. Two types of initial data are investigated in our numerical experiments. First consider 
$$\theta_0(x_1,x_2) = \tanh\left(({2x_2-1})/{0.2}\right)+1,$$ 
which approximates  a bi-color distribution yet satisfies the regularity required in Theorem \ref{thm:optimalitycondition}.
The optimal mixing results with various combinations of cellular flows are displayed in Fig. \ref{fig:mixingonsquare}.
\begin{figure*}[htbp] 
\centering
\includegraphics[width = 0.75\textwidth]{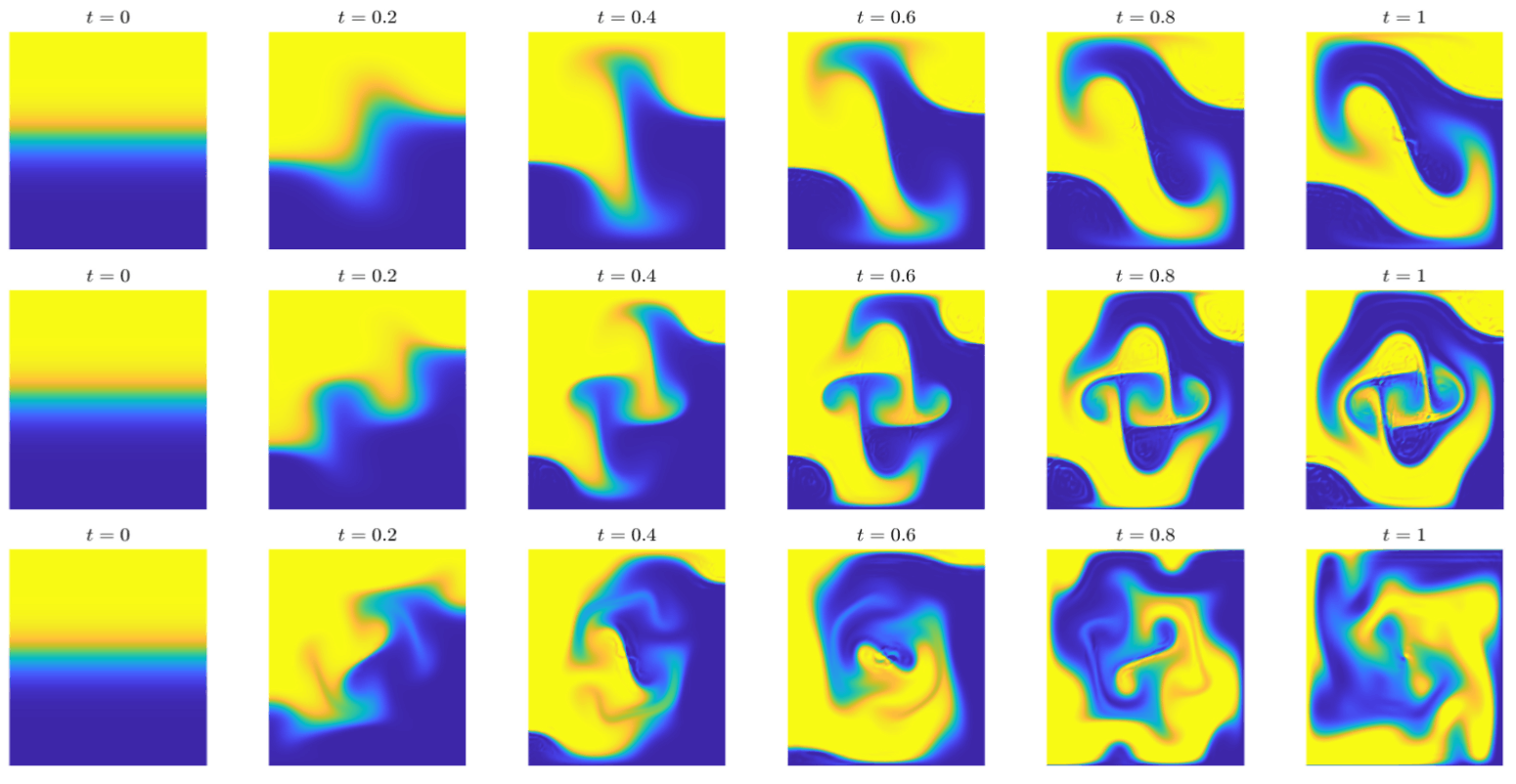}
\caption{Snapshots of the optimal mixing at $t=0,0.2,0.4,0.6,0.8,1$,  for the initial datum  $\theta_0(x_1,x_2) = \tanh((2x_2-1)/0.2)+1$ in the unit square -- \textbf{Top:} $b_1~\&~b_2$; \textbf{Middle:} $b_1~\&~b_3$; \textbf{Bottom:} $b_1~\&~b_4$.}\label{fig:mixingonsquare}
\end{figure*}
As time evolves, the scalar field becomes progressively more mixed as finer and finer structures develop. The detailed quantifications of the system performance under each optimal control are shown in Figure \ref{fig:infoonsquare}. In particular, we compare mixing under constant control inputs versus our optimal strategy.  The constant inputs are  set uniformly as $u_1(t) \equiv 0$ and $u_i(t)\equiv 1$, $i=2,3,4,$ for $t\in [0,1]$,  which only employs  the consistent activation of a single  flow basis function  in each experiment throughout the entire simulation.

The first column of  Fig. \ref{fig:infoonsquare} compares the mix-norms for the constant input case (dashed line), the optimal input case (red solid line), and the target level $rC_0$.  The mix-norm for the constant case shows little decay, whereas the optimal case converges to the target, demonstrating a substantial improvement. 
As shown in the  last two columns of  Fig. \ref{fig:infoonsquare}, the optimal protocol activates both the single-cell flow $u_1(t)$ and multi-cell flows $u_i(t), i=2,3,4,$ achieving effective mixing through their combination. 
It is worth noting that a sign change in the controls indicates a required reversal of flow orientation, as shown in the last plot of the fourth  column.

The optimality of our control design is supported by two results in Fig. \ref{fig:infoonsquare}. First, the relative error of the cost functional converges, as shown in the second column. Second, the complementary slackness condition (see Eq. \eqref{equ:complementaryslackness} in Theorem \ref{thm:optimalitycondition}) is satisfied. Our computed optimal KKT  multiplier is non-zero, indicating an active final-time constraint. This is visually confirmed by the mixing result with the optimal control (red solid line) reaching and tracking the target (dotted line) in the first column.

\begin{figure*}[htbp]
\centering
\includegraphics[width = 0.75\textwidth]{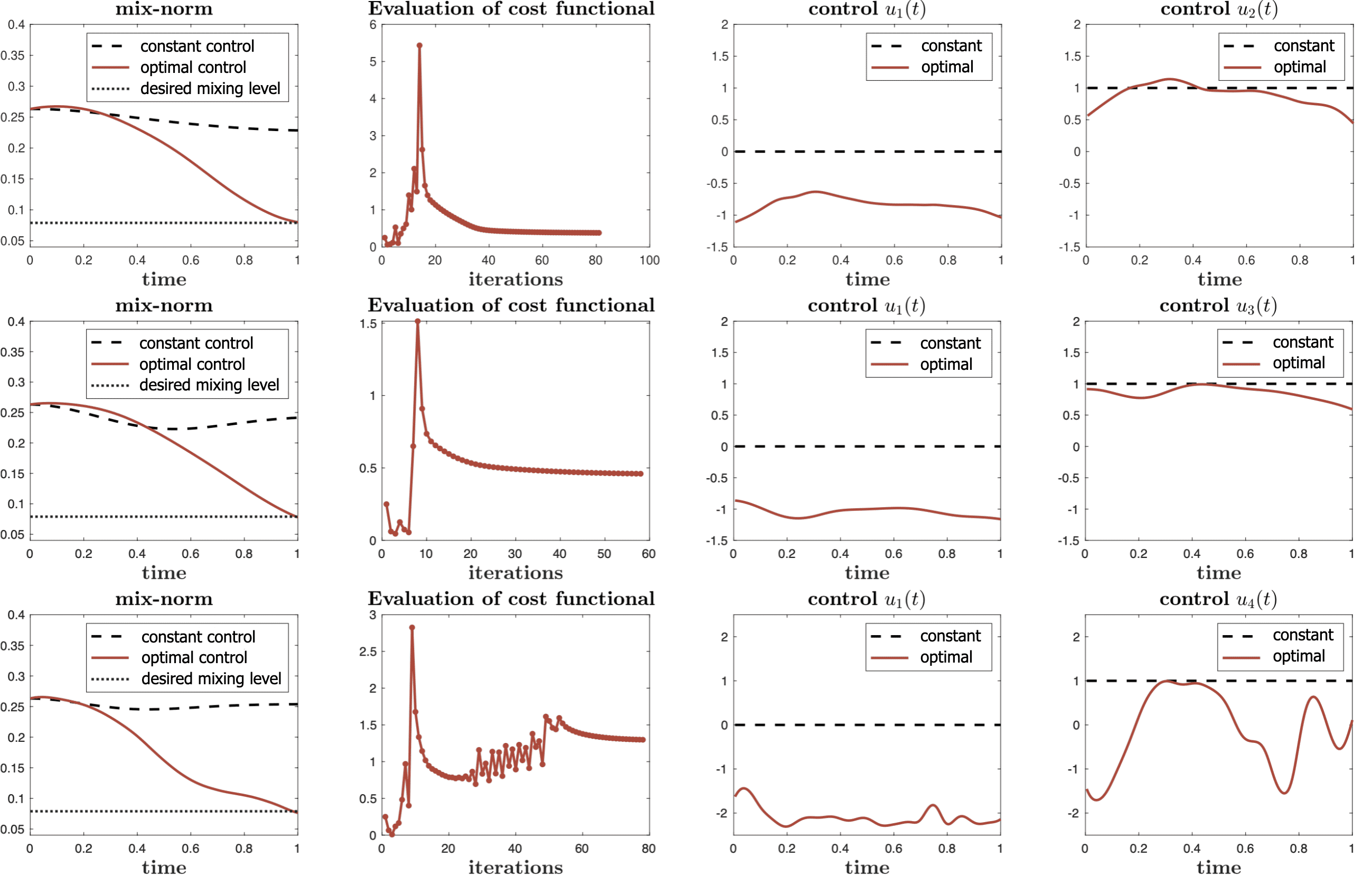}
\caption{Quantification of system performance in the unit square -- \textbf{Top:} $b_1~\&~b_2$; \textbf{Middle:} $b_1~\&~b_3$; \textbf{Bottom:} $b_1~\&~b_4$.}\label{fig:infoonsquare}     
\end{figure*}

We next examine a second  initial distribution given by
$$\theta_0(x_1,x_2) = \sin(2\pi x_2)+1,$$ 
using the same  configuration as in the first case. This yields analogous numerical results, with the mixing outcomes and optimal control profiles displayed in Figs. \ref{fig:mixingonsquareanother}--\ref{fig:infoonsquareanother}. \begin{figure*}[htbp] 
\centering
\includegraphics[width = 0.75\textwidth]{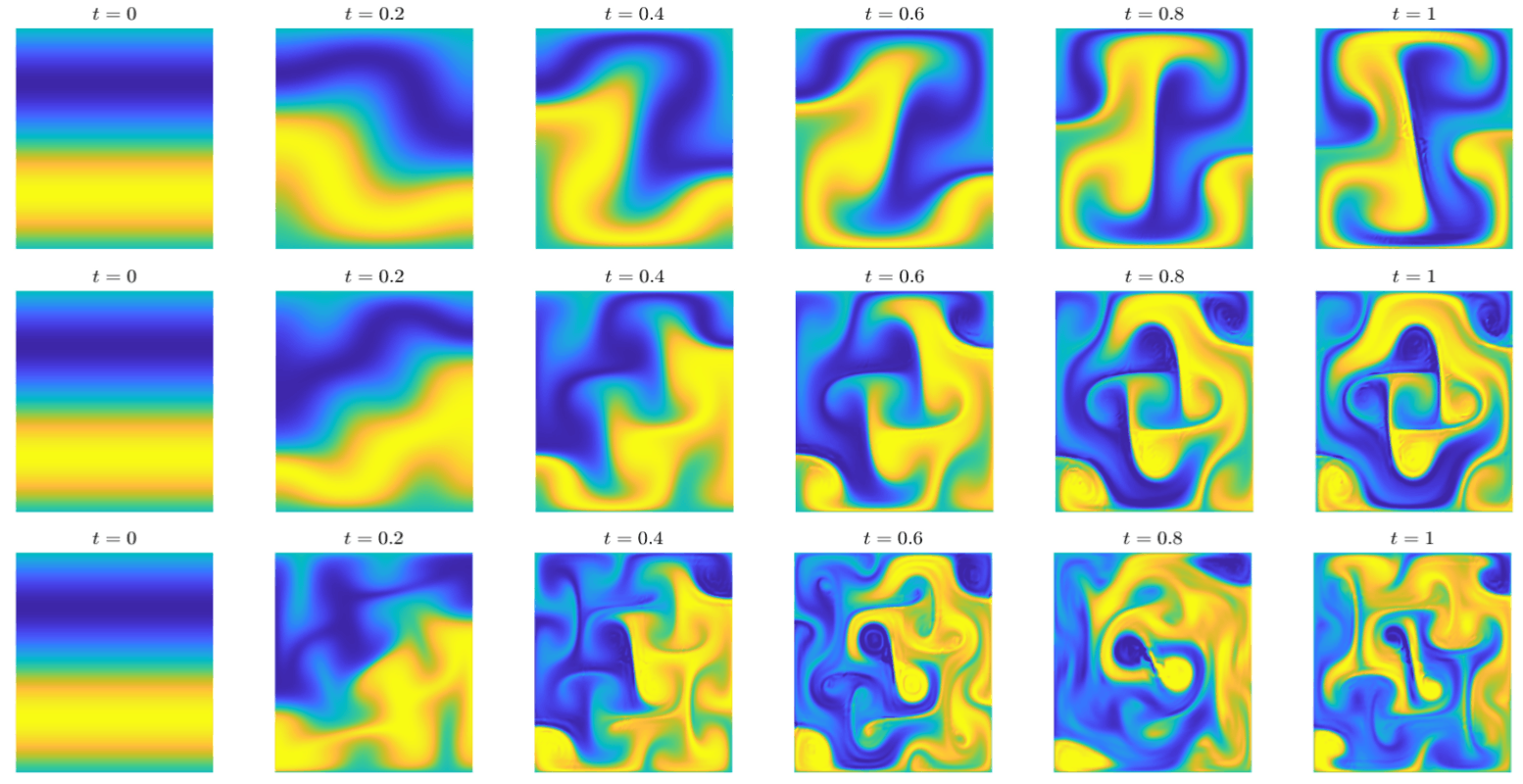}
\caption{Snapshots of the optimal mixing at $t=0,0.2,0.4,0.6,0.8,1,$  for the  initial datum  $\theta_0(x_1,x_2) =  \sin(2\pi x_2)+1$ in the square -- \textbf{Top:} $b_1~\&~b_2$; \textbf{Middle:} $b_1~\&~b_3$; \textbf{Bottom:} $b_1~\&~b_4$.}\label{fig:mixingonsquareanother}
\end{figure*}

\begin{figure*}[htbp]
\centering
\includegraphics[width = 0.75\textwidth]{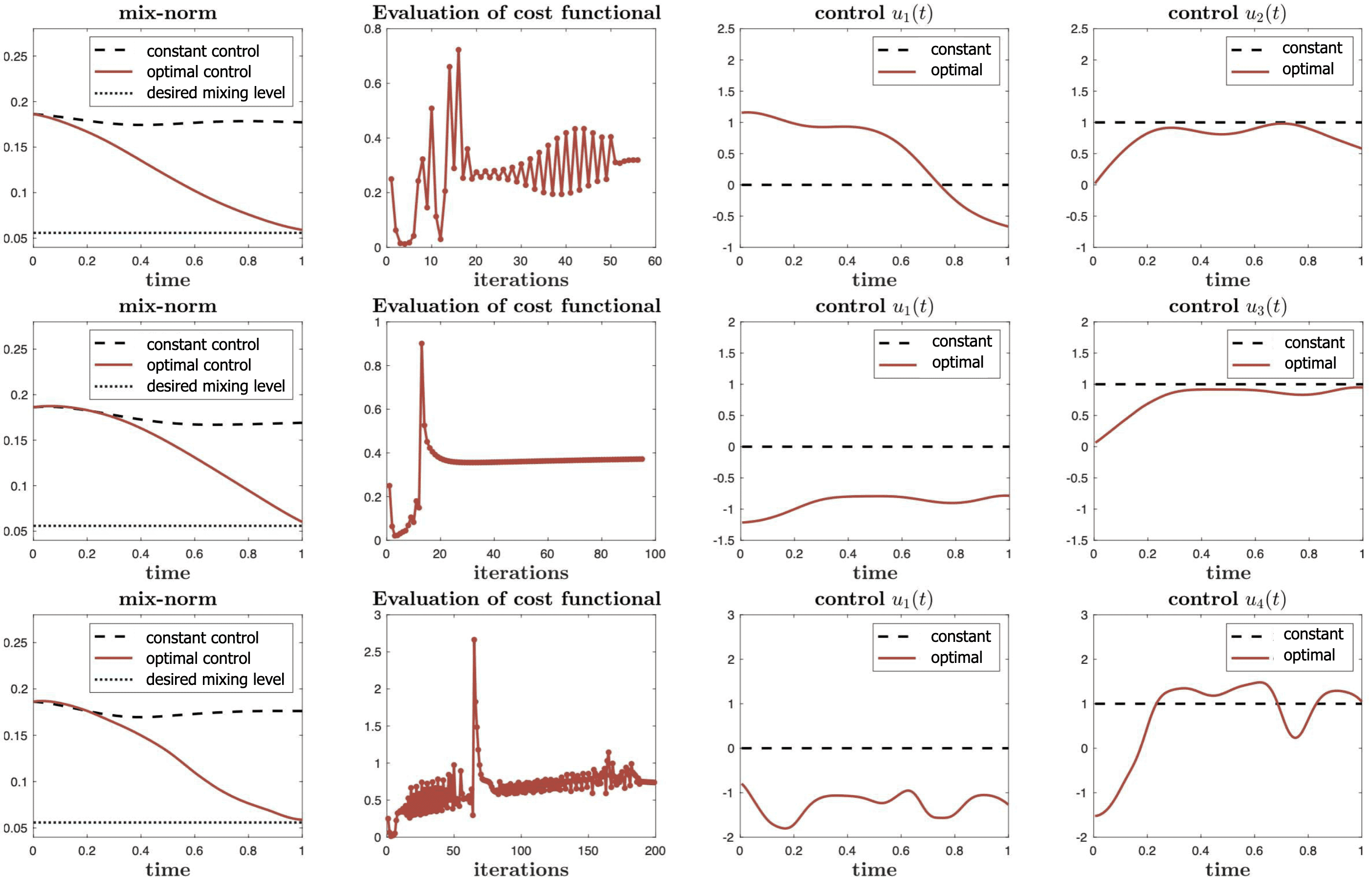}
\caption{Quantification of system performance in the square  -- \textbf{Top:} $b_1~\&~b_2$; \textbf{Middle:} $b_1~\&~b_3$; \textbf{Bottom:} $b_1~\&~b_4$.}\label{fig:infoonsquareanother}
\end{figure*}

Our analysis and numerical implementation can be extended to other domains, provided that cellular flows are suitably defined and satisfy the non-penetration boundary condition in \eqref{equ:basisflows}.

\begin{remark}
The numerical experiments are conducted over a finite time interval and therefore do not provide a rigorous quantitative characterization of mixing rates. Nevertheless, they indicate substantially faster decay of the mix-norm under the optimal time-dependent combinations of cellular flows than under constant or single-mode inputs. In particular, the rapid reduction of the $(H^1(\Omega))'$-norm in Figs.~\ref{fig:infoonsquare} and \ref{fig:infoonsquareanother} is consistent with the formation of finer scalar structures.

The constant and single-mode controls serve as baselines illustrating the geometric limitations of autonomous flows with fixed streamline structures. Allowing multiple cellular fields with different streamline geometries provides additional mechanisms for transport across these invariant structures. However, the presence of multiple modes alone does not guarantee feasibility. As shown in Section~\ref{sec:feasibility}, feasibility depends on the invariant geometry of the selected fields and, for the alternating flow result, on their common kernel property.

Comparisons with prescribed time-periodic inputs, such as oscillatory controls of the form $u_i(t)=\sin(\omega_i t)$, would provide additional benchmarks. Simulations over longer time horizons or for very small values of $r$ would require substantially finer spatial resolution to resolve the resulting filamentation and are therefore left for future work.
\end{remark}


\section{Conclusion}\label{conclusion}

This paper develops a control framework for fluid mixing, drawing inspiration from  the LAP for incompressible flows. A fundamental  departure from the classic framework is the specification of a \emph{constraint set} of acceptable final states in terms of a mix-norm, rather than a single target. This formulation is natural for
mixing applications, as it emphasizes  a desired degree of mixedness  without requiring an exact final configuration. The control velocity is represented by two-dimensional cellular fields that automatically satisfy the incompressibility and no-penetration
conditions.

A central analytical issue is the feasibility of the terminal
constraint under finite dimensional cellular flow actuation.
Using the mixing result of \cite{brue2024enhanced} together with an
amplitude time rescaling of a single cellular flow, we first establish
feasibility for initial data satisfying condition {\rm(H)}. We then
show that alternating between a suitable pair of cellular fields
establishes feasibility for general initial data without condition
{\rm(H)}. We subsequently prove the existence of a global minimizer
by the direct method, derive first-order necessary optimality
conditions, and develop a fixed-point iteration for computing an
optimal control. Numerical simulations demonstrate that the computed
controls satisfy the prescribed mix-norm constraint.

The numerical results further indicate that time-dependent
combinations of different cellular fields can produce more effective
mixing by promoting transport across the streamline structures
associated with the individual modes. However, obtaining uniform mixing rate and control cost estimates for such switching protocols remains open. The quantitative analysis
of these nonautonomous mixing protocols, together with other
Hamiltonian flow designs, constitutes an important direction for future work.


\section{Appendix}\label{appendix}

\subsection{Two remarks on condition {\rm(H)}} \label{appendix_A}

\begin{lemma}\label{lem_equivalence}
For $\theta_0\in H^1(\Omega)$, set
$
q_0:=\theta_0-\bar\theta_0.
$
Then the condition
$
q_0\in\bigl(\ker(b\cdot\nabla)\bigr)^\perp,
$
is
equivalent to requiring that $q_0$ has zero time average along almost
every regular closed streamline of $b$.
\end{lemma}

\begin{proof}
Let $\Phi^s:\Omega\to\Omega$ denote the flow map generated by the
cellular Hamiltonian vector field $b=\nabla^\perp H$, that is,
\[
\frac{d}{ds}\Phi^s(x)
=
b\bigl(\Phi^s(x)\bigr),
\quad
\Phi^0(x)=x.
\]
The associated transport group $S_b(s)$ is defined by
\[
[S_b(s)f](x)
=
f\bigl(\Phi^{-s}(x)\bigr)
=
(f\circ\Phi^{-s})(x),
\quad
f\in L^2(\Omega).
\]
 Since $b$ is smooth and 
incompressible, $\Phi^s$ is measure preserving, and hence $S_b(s)$ is a strongly continuous 
 unitary group on $L^2(\Omega)$.

Define its invariant subspace by
\[
\mathcal{I}_b
:=
\left\{
f\in L^2(\Omega):
S_b(s)f=f
\ \text{for every }s\in\mathbb{R}
\right\}.
\]
Since the generator of $S_b(s)$ is $A_b=-b\cdot\nabla$, with domain
\begin{align*}
  D(A_{b}):=
\big\{
f\in & L^2(\Omega):
b\cdot\nabla f\in L^2(\Omega)\\
& \text{in the sense of distributions}
\big\},
\end{align*}
we have $\mathcal{I}_b
=\ker A_b$, where
\begin{align}
\ker A_b
:=\Big\{
f\in L^2(\Omega):
&\int_\Omega
f\,(b\cdot\nabla\phi)\,dx=0, \nonumber\\
& \qquad \qquad \forall\phi\in H^1(\Omega) 
\Big\}. \label{Def_kernel}
\end{align}
Let $P_b$ denote the $L^2(\Omega)$-orthogonal projection onto
$\mathcal{I}_b$. By the mean ergodic theorem for unitary groups
(e.g., \cite[Section~V.4]{engel2000one}),
\begin{align}
P_bf
=
\lim_{R\to\infty}
\frac{1}{R}\int_0^R S_b(s)f\,ds
\quad\text{in }L^2(\Omega). \label{mean_ergodic}
\end{align}
Because
$
[S_b(s)f](x)=f\bigl(\Phi^{-s}(x)\bigr),
$
the expression on the right hand of \eqref{mean_ergodic} is the long-time average of $f$ along the
streamline through $x$.

Let $x$ lie on a regular closed streamline and $T(x)$ denote its
period. The function
$
s\longmapsto f\bigl(\Phi^{-s}(x)\bigr)
$
is $T(x)$-periodic. The coarea formula ensures that its restriction to
the orbit is integrable for almost every regular streamline.
Therefore, its long-time average equals its average over one period.
It follows that
\[
P_bf(x)
=
\frac{1}{T(x)}
\int_0^{T(x)}
f\bigl(\Phi^s(x)\bigr)\,ds
\]
for almost every $x$ in the regular set. The orientation of the flow
is immaterial because the integral is taken over one complete period.

Since the regular closed streamlines cover $\Omega$ up to a set of
Lebesgue measure zero and $\ker P_b=\mathcal{I}_b^\perp$,
we conclude that
\[
\begin{aligned}
q_0\in\bigl(\ker(b\cdot\nabla)\bigr)^\perp
&\iff q_0\in\mathcal{I}_b^\perp
\iff P_bq_0=0
\\
&\iff
\frac{1}{T(x)}
\int_0^{T(x)}
q_0\bigl(\Phi^s(x)\bigr)\,ds=0
\end{aligned}
\]
along almost every regular closed streamline. This proves the claimed
equivalence.
\end{proof}


The following result  states that reflection-antisymmetric initial data satisfying {\rm(H)}.
\begin{proposition}
\label{prop:symmetric_initial_data}
For the cellular field $b_1=\nabla^\perp H_1$, define the
measure-preserving reflections
\[
\begin{aligned}
R_1(x_1,x_2)&=(1-x_1,x_2),&
R_2(x_1,x_2)&=(x_1,1-x_2),\\
R_3(x_1,x_2)&=(x_2,x_1),&
R_4(x_1,x_2)&=(1-x_2,1-x_1).
\end{aligned}
\]
Let $\theta_0\in H^1(\Omega)$. Suppose that its centered component is
antisymmetric with respect to at least one of these reflections, that
is,
\begin{align}
(\theta_0-\bar\theta_0)\circ R_i
=
-(\theta_0-\bar\theta_0) \label{antisymmetric}
\end{align}
for some $i\in\{1,2,3,4\}$. Equivalently,
$
\theta_0\circ R_i
=
2\bar\theta_0-\theta_0.
$
Then
\[
P_1(\theta_0-\bar\theta_0)=0,
\]
and hence $\theta_0$ satisfies condition {\rm(H)} with respect to
$b_1$.
\end{proposition}

\begin{proof}
Each reflection satisfies
\[
H_1\circ R_i=H_1,
\qquad
DR_i(x)b_1(x)=-b_1(R_i(x)),
\]
where $DR_i$ denotes the Jacobian matrix of $R_i$. Let $\Phi_1^s$
denote the flow generated by $b_1$. The second identity implies that
\[
R_i\circ\Phi_1^s
=
\Phi_1^{-s}\circ R_i.
\]
Indeed, the curve $s\mapsto R_i(\Phi_1^s(x))$ solves the equation
generated by $-b_1$ with initial value $R_i(x)$.

Moreover, $H_1\circ R_i=H_1$, so $R_i$ maps each level set of $H_1$
onto itself. Because every regular level set of $H_1$ in $\Omega$ is
a single closed orbit, $x$ and $R_i(x)$ lie on the same orbit. Hence,
\[
\begin{aligned}
P_1(\phi\circ R_i)(x)
&=
\frac{1}{T(x)}
\int_0^{T(x)}
\phi\bigl(R_i(\Phi_1^s(x))\bigr)\,ds
\\
&=
\frac{1}{T(x)}
\int_0^{T(x)}
\phi\bigl(\Phi_1^{-s}(R_i(x))\bigr)\,ds
\\
&=
P_1\phi(R_i(x))
=
P_1\phi(x).
\end{aligned}
\]
The last equality follows because $P_1\phi$ is constant along each
closed orbit. Therefore,
\begin{align}
  P_1(\phi\circ R_i)=P_1\phi, \quad \phi\in L^2(\Omega).\label{antisym_R}
\end{align}
Set $q_0=\theta_0-\bar\theta_0$. By \eqref{antisymmetric}--\eqref{antisym_R}, we obtain
\[
P_1q_0
=
P_1(q_0\circ R_i)
=
P_1(-q_0)
=
-P_1q_0.
\]
Consequently, $P_1(q_0)=0$. By Lemma \ref{lem_equivalence}, $\theta_0$ satisfies {\rm(H)} with respect
to $b_1$.
\end{proof}

To obtain a nonnegative initial datum, we just replace $\theta_0$ by
$\theta_0+c$ for
$c\geq\|\theta_0\|_{L^\infty(\Omega)}$.

\subsection{Feasibility of switching controls} \label{appendix_B}

To prove Theorem~\ref{thm:switching_feasibility}, we first establish
three auxiliary lemmas. The first shows that, under a fixed cellular
flow $b_k$, the transported scalar converges in the mix-norm to its
$L^2(\Omega)$-orthogonal projection onto the invariant subspace of
the flow.

\begin{lemma}\label{lem:projection_limit}
For every $k\in\mathbb{N}$ and $f\in L^2(\Omega)$,
\[
S_k(s)f\longrightarrow P_kf
\quad\text{strongly in }\mathcal{X}
\quad\text{as }s\to\infty,
\]
where $S_k(s)$ is the transport group generated by $A_k=-b_k\cdot\nabla$ and $P_k$
is the $L^2(\Omega)$-orthogonal projection onto
\[
K_k:=\ker A_k=\ker(b_k\cdot\nabla),
\]
as in Lemma~\ref{lem_equivalence}.
\end{lemma}

\begin{proof}
We first consider a smooth function of the form
\[
q=A_k\psi=-b_k\cdot\nabla\psi,\qquad
\psi\in C^\infty(\overline\Omega).
\]
Since $A_k$ is skew-adjoint,
\[
q\in\operatorname{Ran}A_k\subset(\ker A_k)^\perp=K_k^\perp.
\]
By Lemma~\ref{lem_equivalence}, $q$ has zero invariant time average
along almost every regular closed streamline of $b_k$.

Moreover, for each fixed \(k\), the square \(\Omega\) is the union of the \(k^2\) invariant cells
 $$
 Q_m =
 \prod_{j=1}^2
 \left(\frac{m_j}{k},\frac{m_j+1}{k}\right),
 \
 m=(m_1,m_2)\in\{0,\ldots,k-1\}^2.
 $$
 On each \(Q_m\), the change of variables \(y=kx-m\) maps \(Q_m\) onto \((0,1)^2\) and transforms \(b_k\) into the standard cellular field, up to the sign \((-1)^{m_1+m_2}\) and a constant rescaling of time. Therefore,  \eqref{def_mixing_rate}--\eqref{stability_one_flow} apply on each cell. Summing over the finitely many cells gives 
 $$
  \|S_k(s)q\|_{\mathcal X}\longrightarrow0, \quad\text{as }s\to\infty. 
 $$
 This argument is carried out separately for each fixed $k$; no
mixing-rate estimate uniform in $k$ is needed.

Since $A_k$ is skew-adjoint on $L^2(\Omega)$ and smooth functions form
a core for its generator,
$$
\overline{A_k\bigl(C^\infty(\overline\Omega)\bigr)}^{\,L^2}
=
\overline{\operatorname{Ran}A_k}^{\,L^2}
=
(\ker A_k^*)^\perp
=
K_k^\perp.
$$
Hence, for every $g\in K_k^\perp$, there exist
$q_n=A_k\psi_n$, with $\psi_n\in C^\infty(\overline\Omega)$, such that
$q_n\to g$ in $L^2(\Omega)$. Since $S_k(s)$ is an $L^2$-isometry and
the embedding $L^2(\Omega)\hookrightarrow\mathcal X$ is continuous,
$$
\begin{aligned}
\|S_k(s)g\|_{\mathcal X}
&\leq
\|S_k(s)(g-q_n)\|_{\mathcal X}
+\|S_k(s)q_n\|_{\mathcal X}\\
&\leq
C_\Omega\|g-q_n\|_{L^2(\Omega)}
+\|S_k(s)q_n\|_{\mathcal X}.
\end{aligned}
$$
First letting $s\to\infty$ and then $n\to\infty$ yields
$$
S_k(s)g\longrightarrow0
\qquad\text{strongly in }\mathcal X.
$$

Finally, for arbitrary $f\in L^2(\Omega)$, write
$$
f=P_kf+(I-P_k)f,
\qquad
(I-P_k)f\in K_k^\perp.
$$
Since $S_k(s)P_kf=P_kf$, we conclude that
$$
S_k(s)f-P_kf
=
S_k(s)(I-P_k)f
\longrightarrow0
\quad\text{strongly in }\mathcal X.
$$
\end{proof}

The following lemma verifies the common kernel condition in
Theorem~\ref{thm:switching_feasibility} for the pair $(b_1,b_2)$,
thereby showing that the hypothesis is nonvacuous. The condition,
however, is not automatic for every pair of cellular flows.

\begin{lemma}\label{lem:common_kernel}
For the cellular fields $b_1$ and $b_2$ defined in
\eqref{cellular_2D}, we have 
\begin{align*}
  K_1\cap K_2=\operatorname{span}\{1\}. 
\end{align*}
\end{lemma}

\begin{proof}
A direct calculation gives
\[
\det(b_1,b_2)
=
2\sin(\pi x_1)\sin(\pi x_2)
\left(
\cos^2(\pi x_2)-\cos^2(\pi x_1)
\right).
\]
Therefore, in the interior of $\Omega$,
\[
\det(b_1,b_2)=0
\quad\Longleftrightarrow\quad
x_1=x_2
\quad\text{or}\quad
x_1+x_2=1.
\]
Let
\[
\Sigma
:=
\{x_1=x_2\}\cup\{x_1+x_2=1\}.
\]
The set $\Omega\setminus\Sigma$ has four connected components, and
$b_1,b_2$ span $\mathbb R^2$ at every point of each component.

Suppose that $f\in K_1\cap K_2$. Then
\[
b_1\cdot\nabla f=b_2\cdot\nabla f=0
\]
in the sense of distributions. On each connected component of
$\Omega\setminus\Sigma$, the matrix having $b_1$ and $b_2$ as its
columns is smoothly invertible. Hence
$
\nabla f=0
$
 in the distributional sense, and $f$ is constant on each of
the four components.

It remains to show that these four constants agree. Along $x_1=x_2$, the normal component of $b_1$ is, up to a nonzero constant factor, 
\begin{align*}
b_1\cdot \vec{n}|_{x_1=x_2}
=b_1(x_1,x_1)\cdot(1,-1)
=2\sin(\pi x_1)\cos(\pi x_1),
\end{align*}
whereas along $x_1+x_2=1$, it is
\[
b_1\cdot \vec{n}|_{x_1+x_2=1}=b_1(x_1,1-x_1)\cdot(1,1)
=
-2\sin(\pi x_1)\cos(\pi x_1).
\]
These quantities are nonzero for $x_1\in (0,1)\setminus\{\frac{1}{2}\}$.
Therefore, a jump of $f$ across either diagonal would produce a
nonzero measure-valued term in $b_1\cdot\nabla f$, contradicting
$b_1\cdot\nabla f=0$. The constants on adjacent components must
thus coincide, proving that $f$ is constant almost everywhere in
$\Omega$.
\end{proof}

For later use, observe that for each fixed $s$, the operator
$S_k(s)$ is bounded on $\mathcal X$. Indeed, if $\Phi_k^s$ is the
flow of $b_k$, then, by duality,
\[
\langle S_k(s)f,\varphi\rangle
=
\langle f,\varphi\circ\Phi_k^s\rangle.
\]
Since $\Phi_k^s$ is a $C^1$ volume-preserving diffeomorphism of
$\Omega$,
\begin{align}
\|\varphi\circ\Phi_k^s\|_{H^1}
\leq C_k(s)\|\varphi\|_{H^1}, \label{EST_S}
\end{align}
and consequently
\[
\|S_k(s)f\|_{\mathcal X}
\leq C_k(s)\|f\|_{\mathcal X}.
\]

The next lemma connects the ideal projection dynamics to the
controlled transport equation by approximating any finite product of
projections with a finite sequence of cellular flow evolutions.

\begin{lemma}\label{lem:realize_projections}
Let $i_1,\ldots,i_m\in\{1,\ldots,N\}$. For every
$f\in L^2(\Omega)$ and every $\zeta>0$, there exist finite durations
$s_1,\ldots,s_m>0$ such that
\begin{align}
\left\|
S_{i_m}(s_m)\cdots S_{i_1}(s_1)f
-
P_{i_m}\cdots P_{i_1}f
\right\|_{\mathcal X}
<\zeta. \label{converge}
\end{align}
\end{lemma}

\begin{proof}
For each $j=1,\ldots,m$, let
\[
U_j(s):=S_{i_j}(s),\quad s\geq0,
\quad\text{and}\quad
Q_j:=P_{i_j}.
\]
We shall choose durations $s_m,\ldots,s_1$ recursively in reverse
order and then set
\[
U_j:=U_j(s_j)=S_{i_j}(s_j).
\]
The following telescoping identity plays an important role in establishing our claim:
\begin{align}
&U_m\cdots U_1-Q_m\cdots Q_1\nonumber\\
&\quad=
\sum_{j=1}^m
U_m\cdots U_{j+1}
(U_j-Q_j)
Q_{j-1}\cdots Q_1, \label{UQ_identity}
\end{align}
where for $j=1$, the product $Q_{j-1}\cdots Q_1$ is replaced by identity;
and
for $j=m$, since no transport operator acts after $U_m$, the product $U_m\cdots U_{j+1}$ is also replaced by identity.  
To verify \eqref{UQ_identity}, define
\[
R_j:=U_m\cdots U_{j+1}Q_j\cdots Q_1,
\qquad j=0,\ldots,m,
\]
so that $R_0=U_m\cdots U_1$ and
$R_m=Q_m\cdots Q_1$. Then
\[
R_{j-1}-R_j
=
U_m\cdots U_{j+1}(U_j-Q_j)Q_{j-1}\cdots Q_1.
\]
Summing this identity from $j=1$ to $m$ yields \eqref{UQ_identity}. 
Essentially, the identity separates the total approximation error into \(m\) terms, each containing exactly one individual error \(U_j-Q_j\), which is helpful in the rest of our proof.

Fix $m$, a switching sequence
$i_1,\ldots,i_m\in\{1,\ldots,N\}$ and $f\in L^2(\Omega)$.
For $j=1,\ldots,m$, set
\[
g_j:=Q_{j-1}\cdots Q_1f,
\]
where $g_1=f$. Thus, each $g_j$ is a fixed element of
$L^2(\Omega)$.

We determine the time durations in reverse order, beginning with the
final stage $j=m$. Since   $U_m\cdots U_{j+1}=I$, the
corresponding error term on the right hand of  \eqref{UQ_identity}, applied to $f$,  becomes 
\[
(U_m-Q_m)g_m, \qquad
g_m:=Q_{m-1}\cdots Q_1f.
\]

By
Lemma~\ref{lem:projection_limit}, we know that 
\[
\bigl(S_{i_m}(s)-P_{i_m}\bigr)g_m
\longrightarrow0
\quad\text{in }\mathcal X
\]
as $s\to\infty$. Hence $s_m$ can be chosen so that
\[
\|(U_m-Q_m)g_m\|_{\mathcal X}<\frac{\zeta}{m}.
\]

Suppose now that $s_m,\ldots,s_{j+1}$ have already been chosen.
Then by \eqref{EST_S}
\[
V_j:=U_m\cdots U_{j+1}=S_{i_m}(s_m)\cdots S_{i_{j+1}}(s_{j+1})
\]
is a fixed bounded operator on $\mathcal X$. Set
\begin{align}
C_j:=\|V_j\|_{\mathcal L(\mathcal X)}<\infty. \label{EST_V}
\end{align}
Since $g_j$ is fixed, Lemma~\ref{lem:projection_limit} allows us to
choose $s_j$ sufficiently large that
\[
\|(U_j-Q_j)g_j\|_{\mathcal X}
<
\frac{\zeta}{mC_j}.
\]
By \eqref{EST_V}, it follows that
\[
\|V_j(U_j-Q_j)g_j\|_{\mathcal X}
<
\frac{\zeta}{m},
\]
that is,
\[
\left\|
U_m\cdots U_{j+1}
(U_j-Q_j)
Q_{j-1}\cdots Q_1f
\right\|_{\mathcal X}
<\frac{\zeta}{m}.
\]
Proceeding backward from $j=m$ to $j=1$ until $s_1$ is chosen. Then  applying the telescoping
identity \eqref{UQ_identity} establishes  the desired  result  \eqref{converge}.

\end{proof}

We now combine the preceding projection, common kernel, and
finite product approximation results to construct an admissible
switching control satisfying the terminal constraint.

\begin{proof}[Proof of Theorem~\ref{thm:switching_feasibility}] \label{feasibility_proof}
Let $\theta_0\in H^1(\Omega)$ and set 
\[
q_0:=\theta_0-\bar\theta_0,
\qquad
C_0:=\|q_0\|_{\mathcal X}.
\]
If $C_0=0$, the zero control satisfies the terminal constraint.
Hence, assume $C_0>0$.

Let $P_I$ and $P_J$ be the orthogonal projections onto $K_I$ and
$K_J$, respectively. By the von Neumann alternating-projection
theorem \cite[Lemma 22, p.\,475]{von1949rings},
\[
(P_JP_I)^nq_0
\longrightarrow
P_{K_I\cap K_J}q_0
=
P_{\operatorname{span}\{1\}}q_0
=
0
\]
strongly in $L^2(\Omega)$, because $q_0$ has zero spatial mean, 
and therefore also in $\mathcal X$.

Now choose $n$ sufficiently large such that
\begin{align}
  \|(P_JP_I)^nq_0\|_{\mathcal X}<\frac{rC_0}{2}.  \label{EST_Pn}
\end{align}
Let $m=2n$ and define the alternating sequence 
\[
i_{2\ell-1}=I,
\qquad
i_{2\ell}=J,
\qquad \ell=1,\ldots,n.
\]
By Lemma~\ref{lem:realize_projections}, there exist finite durations
$s_1,\ldots,s_m>0$ such that
\begin{align}
\left\|
S_{i_m}(s_m)\cdots S_{i_1}(s_1)q_0
-
(P_JP_I)^nq_0
\right\|_{\mathcal X}
<\frac{rC_0}{2}.\label{EST_SP_n}
\end{align}
Consequently,  combining \eqref{EST_SP_n} with \eqref{EST_Pn} yields 
\begin{align}
\|S_{i_m}(s_m)\cdots S_{i_1}(s_1)q_0\|_{\mathcal X}
<rC_0.  \label{EST_product_S}
\end{align}

It remains to realize these autonomous evolution times within the
prescribed interval $[0,t_f]$. Set
\[
\delta:=\frac{t_f}{m},
\qquad
I_j:=((j-1)\delta,j\delta],
\]
and on $I_j$, define
\[
u_{i_j}(t):=\frac{s_j}{\delta},
\qquad
u_\ell(t):=0\quad\text{for }\ell\neq i_j.
\]
The resulting piecewise constant control activates $b_I$ and $b_J$
in alternating order. Since $m<\infty$, $\delta>0$, and every
duration $s_j$ chosen above is finite,
\[
\|u\|_{L^\infty(0,t_f;\mathbb R^N)}
=
\max_{1\leq j\leq m}\frac{s_j}{\delta}
<\infty.
\]
Moreover, because $|I_j|=\delta$,
\[
\int_0^{t_f}|u(t)|^2\,dt
=
\sum_{j=1}^m
\int_{I_j}\left(\frac{s_j}{\delta}\right)^2dt
=
\sum_{j=1}^m\frac{s_j^2}{\delta}
<\infty.
\]
Therefore, $u\in U_{\mathrm{ad}}$.
The corresponding solution satisfies
\[
\theta(t_f)
=
S_{i_m}(s_m)\cdots S_{i_1}(s_1)\theta_0.
\]
Since every constant function is invariant under each transport
group, we have
\[
\theta(t_f)-\bar\theta_0
=
S_{i_m}(s_m)\cdots S_{i_1}(s_1)q_0.
\]
As a result of \eqref{EST_product_S}, we obtain that 
\[
\|\theta(t_f)-\bar\theta_0\|_{\mathcal X}
<
r\|\theta_0-\bar\theta_0\|_{\mathcal X},
\]
which proves the terminal constraint and the nonemptiness of the
feasible set.
\end{proof}

Since $r>0$ can be chosen arbitrarily small, the theorem establishes approximate controllability to the mean in the mix-norm, which we term \emph{approximate mixing controllability}.

 \section*{Acknowledgement}
 W. Hu is partially supported by NSF DMS-2111486, DMS-2205117 and AFOSR FA9550-23-1-0675. This work was also  under the support of Humboldt Research Fellowship for Experienced Researchers program from the Alexander von Humboldt Foundation, during W. Hu's visit at the Chair for Dynamics, Control, Machine Learning and Numerics, Friedrich-Alexander-Universit\" at Erlangen N\"urnberg, Germany.

\bibliographystyle{IEEEtran} 
\bibliography{myref}

\end{document}